\documentclass[11pt]{article}
\usepackage[margin=1in]{geometry} 
\usepackage{amssymb,amsfonts,amsmath,bbm,mathrsfs,stmaryrd}
\usepackage{xcolor}
\usepackage{url}

\usepackage{enumerate}

\usepackage{hyperref}
\hypersetup{colorlinks,
             linkcolor=black!75!red,
             citecolor=blue,
             pdftitle={},
             pdfproducer={pdfLaTeX},
             pdfpagemode=None,
             bookmarksopen=true
             bookmarksnumbered=true}

\usepackage{tikz}
\usetikzlibrary{arrows,calc,decorations.pathreplacing,decorations.markings,intersections,shapes.geometric,through,fit,shapes.symbols,positioning,decorations.pathmorphing}

\usepackage{braket}

\usepackage[amsmath,thmmarks,hyperref]{ntheorem}
\usepackage{cleveref}

\creflabelformat{enumi}{#2(#1)#3}

\crefname{section}{Section}{Sections}
\crefformat{section}{#2Section~#1#3} 
\Crefformat{section}{#2Section~#1#3} 

\crefname{subsection}{\S}{\S\S}
\AtBeginDocument{%
  \crefformat{subsection}{#2\S#1#3}%
  \Crefformat{subsection}{#2\S#1#3}%
}

\crefname{subsubsection}{\S}{\S\S}
\AtBeginDocument{%
  \crefformat{subsubsection}{#2\S#1#3}%
  \Crefformat{subsubsection}{#2\S#1#3}%
}

\theoremstyle{plain}

\newtheorem{lemma}{Lemma}[section]
\newtheorem{proposition}[lemma]{Proposition}

\newtheorem{theorem}[lemma]{Theorem}

\theoremstyle{nonumberplain}

\theoremstyle{plain}
\theorembodyfont{\upshape}
\theoremsymbol{\ensuremath{\blacklozenge}}

\newtheorem{example}[lemma]{Example}

\crefname{definition}{definition}{definitions}
\crefformat{definition}{#2definition~#1#3} 
\Crefformat{definition}{#2Definition~#1#3} 

\crefname{ex}{example}{examples}
\crefformat{example}{#2example~#1#3} 
\Crefformat{example}{#2Example~#1#3} 

\crefname{remark}{remark}{remarks}
\crefformat{remark}{#2remark~#1#3} 
\Crefformat{remark}{#2Remark~#1#3} 

\crefname{convention}{convention}{conventions}
\crefformat{convention}{#2convention~#1#3} 
\Crefformat{convention}{#2Convention~#1#3} 

\crefname{notation}{notation}{notations}
\crefformat{notation}{#2notation~#1#3} 
\Crefformat{notation}{#2Notation~#1#3} 

\crefname{table}{table}{tables}
\crefformat{table}{#2table~#1#3} 
\Crefformat{table}{#2Table~#1#3}

\crefname{lemma}{lemma}{lemmas}
\crefformat{lemma}{#2lemma~#1#3} 
\Crefformat{lemma}{#2Lemma~#1#3} 

\crefname{proposition}{proposition}{propositions}
\crefformat{proposition}{#2proposition~#1#3} 
\Crefformat{proposition}{#2Proposition~#1#3} 

\crefname{corollary}{corollary}{corollaries}
\crefformat{corollary}{#2corollary~#1#3} 
\Crefformat{corollary}{#2Corollary~#1#3} 

\crefname{theorem}{theorem}{theorems}
\crefformat{theorem}{#2theorem~#1#3} 
\Crefformat{theorem}{#2Theorem~#1#3} 

\crefname{enumi}{}{}
\crefformat{enumi}{(#2#1#3)}
\Crefformat{enumi}{(#2#1#3)}

\crefname{assumption}{assumption}{Assumptions}
\crefformat{assumption}{#2assumption~#1#3} 
\Crefformat{assumption}{#2Assumption~#1#3} 

\crefname{equation}{}{}
\crefformat{equation}{(#2#1#3)} 
\Crefformat{equation}{(#2#1#3)}

\numberwithin{equation}{section}

\theoremstyle{nonumberplain}
\theoremsymbol{\ensuremath{\blacksquare}}

\newtheorem{proof}{Proof}
\newcommand\pf[1]{\newtheorem{#1}{Proof of \Cref{#1}}}

\newcommand\bR{{\mathbb R}}
\newcommand\bS{{\mathbb S}}
\newcommand\bT{{\mathbb T}}

\newcommand\bZ{{\mathbb Z}}

\newcommand\fa{{\mathfrak a}}

\newcommand\fg{{\mathfrak g}}
\newcommand\fh{{\mathfrak h}}

\newcommand\fm{{\mathfrak m}}

\newcommand{\qedhere}{\mbox{}\hfill\ensuremath{\blacksquare}}

\title{A characteristic-index inequality for closed embeddings of locally compact groups}
\author{Alexandru Chirvasitu}

\begin{document}

\date{}

\newcommand{\Addresses}{{
  \bigskip
  \footnotesize

  \textsc{Department of Mathematics, University at Buffalo, Buffalo,
    NY 14260-2900, USA}\par\nopagebreak \textit{E-mail address}:
  \texttt{achirvas@buffalo.edu}

}}

\maketitle

\begin{abstract}
  The characteristic index of a locally compact connected group $G$ is the non-negative integer $d$ for which we have a homeomorphism $G\cong K\times \mathbb{R}^d$ with $K\le G$ maximal compact. We prove that the characteristic indices of closed connected subgroups are dominated by those of the ambient groups.
\end{abstract}


\vspace{.5cm}



\section*{Introduction}

The {\it characteristic index} (here denoted by $\mathrm{ci}(G)$) of a locally compact connected group $G$ was introduced in \cite{iw} (Theorem 13 therein) in the course of analyzing the structure of such groups: one can always find a closed submanifold $\bR^d\subseteq G$ such that, for a maximal compact subgroup $K\le G$
\begin{equation*}
  \text{multiplication}: K\times \bR^d\to G 
\end{equation*}
is a homeomorphism. Now simply define the characteristic index by
\begin{equation*}
  \mathrm{ci}(G):=d;
\end{equation*}
a measure, in other words, of ``how non-compact'' $G$ is. It is not difficult to see that this definition is sound: given homeomorphisms
\begin{equation*}
  K\times \bR^d\cong G\cong K'\times \bR^{d'}
\end{equation*}
as above, for maximal compact subgroups $K$ and $K'$, one can
\begin{itemize}
\item assume that $K=K'$ because all maximal compact subgroups are mutually conjugate (e.g. \cite[Theorem 13]{iw} and \cite[\S 4.13, first Theorem]{mz});
\item and that $G$ is Lie by substituting $K/N\le G/N$ for $K\le G$ for a compact normal subgroup $N\trianglelefteq G$ with $G/N$ Lie \cite[\S 4.6, Theorem]{mz};
\item hence affording a dimension count:
  \begin{equation*}
    d = \dim G - \dim K = d'.
  \end{equation*}
\end{itemize}

The characteristic index was useful recently in \cite{2107.11796v1} for the purpose of studying colimits in the category of locally compact groups. That analysis required an understanding of how characteristic indices behave under locally-compact-group morphisms:
\begin{enumerate}[(a)]
\item\label{item:3} they can only {\it decrease} along dense-image morphisms $f:H\to G$ \cite[Theorem 2.3]{2107.11796v1} (a slight generalization of \cite[Lemma 4.10]{iw}, the analogue for quotients by closed normal subgroups)
\item\label{item:4} while on the other hand, they can only {\it increase} along closed embeddings $H\to G$ with $H$ Lie and semisimple \cite[Proposition 2.4]{2107.11796v1}.
\end{enumerate}
This latter result, in particular, while sufficient as auxiliary material for \cite{2107.11796v1}, leaves open the natural question of whether $\mathrm{ci}(H)\le \mathrm{ci}(G)$ in full generality for any closed embedding $H\le G$ of connected Locally compact groups. The aim of the present note is to prove that this is indeed the case (\Cref{th:cihg}):

\begin{theorem}\label{th:mainintro}
  For a closed embedding $H\le G$ of connected locally compact groups the characteristic index of $G$ dominates that of $H$. \qedhere
\end{theorem}

In part, the reason why this appears not to be as straightforward as one might hope is the tension between the two phenomena \Cref{item:3} and \Cref{item:4} above: as \Cref{ex:hgh} makes clear, it is possible to
\begin{itemize}
\item start with a (semisimple, say) group $H$;
\item enlarge its characteristic index by taking a product with a Euclidean group $N\cong \bR^d$;
\item and then bring the characteristic index back down (as in \Cref{item:3}) through a dense embedding $HN\le \overline{HN}=G$. 
\end{itemize}
The example in question shows that this last step can shrink $\mathrm{ci}(G)$ all the way back down to $\mathrm{ci}(H)$, but the point of \Cref{th:mainintro} is that so long as $H\le G$ is closed there can be no {\it further} characteristic-index shrinkage.

\subsection*{Acknowledgements}

This work was supported in part by NSF grant DMS-2001128.

\section{Preliminaries}\label{se.prel}

The topological groups considered here are all Hausdorff. In fact, being $T_0$-separated (\cite[\S 1.1, Definition, condition 4)]{mz}) will do: being Hausdorff follows \cite[\S 1.16]{mz} (in the context of topological {\it groups}, not for arbitrary spaces), along with complete regularity \cite[\S 1.18]{mz}.

We record the following observation on the behavior of characteristic indices under passage to quotients, which aggregates a couple of results in the literature (on quotients by connected / discrete subgroups respectively).

\begin{lemma}\label{le:diff}
  Let $G$ be a connected locally compact group and $N\trianglelefteq G$ a closed normal subgroup with identity connected component $N_0$. We then have
  \begin{equation*}
    \mathrm{ci}(G) = \mathrm{ci}(N_0) + \mathrm{rank}(N/N_0) + \mathrm{ci}(G/N). 
  \end{equation*}
  where $N/N_0$ is finitely-generated abelian and its rank is the largest $r$ for which there is an embedding $\bZ^r\le N/N_0$.
\end{lemma}
\begin{proof}
  We first handle the quotient by $N_0$:
  \begin{equation}\label{eq:gn0}
    \mathrm{ci}(G) = \mathrm{ci}(N_0) + \mathrm{ci}(G/N_0) 
  \end{equation}
  by \cite[Lemma 4.10]{iw} (and \cite[\S 4.6, Theorem]{mz}, which ensures that the hypothesis of that lemma is met by connected locally compact groups).

  Substituting $N/N_0\trianglelefteq G/N_0$, we can now assume that the normal subgroup ($N/N_0$, in this case) is discrete. But then it will be finitely-generated abelian (by \cite[Lemma 2.1]{2107.11796v1}, for instance), and hence \cite[Proposition 0.3]{2107.11796v1} applies to prove
  \begin{equation}\label{eq:gnn0}
    \mathrm{ci}(G/N_0) = \mathrm{rank}(N/N_0) + \mathrm{ci}(G/N).
  \end{equation}
  Combining \Cref{eq:gn0,eq:gnn0} gives the desired result.
\end{proof}

\section{The main result}\label{se.main}

As mentioned, we are after

\begin{theorem}\label{th:cihg}
  For any closed embedding $H\le G$ of connected locally compact groups we have
  \begin{equation}\label{eq:cihg}
    \mathrm{ci}(H)\le \mathrm{ci}(G).
  \end{equation}  
\end{theorem}

Before embarking on the proof, a remark on what will {\it not} work. Suppose that in \Cref{th:cihg} we have restricted attention to Lie groups (as we will). Since a connected Lie group is analytically isomorphic to a manifold of the form
\begin{equation*}
  G\cong (\text{maximal compact subgroup})\times \bR^{\mathrm{ci}(G)}
\end{equation*}
(as follows, for instance, from \cite[Theorem 13]{iw} or \cite[\S 4.13, first Theorem]{mz}), one might hope that \Cref{th:cihg} would follow from a general result to the effect that for any analytic closed embedding
\begin{equation}\label{eq:analyticmn}
  (\text{compact analytic manifold})\times \bR^m \subseteq (\text{compact analytic manifold})\times \bR^n
\end{equation}
we have to have $m\le n$. This is not the case:

\begin{example}\label{ex:r21}
  Consider the closed analytic embedding
  \begin{equation*}
    \bR^2\ni (x,\ y)\mapsto (\varphi(x,y),\ x^2+y^2) \in \bS^2\times \bR,
  \end{equation*}
  where $\varphi:\bR^2\to \bS^2$ is the inverse of the stereographic projection \cite[Problem 1-7]{lee}, mapping $\bR^2$ isomorphically onto the complement of the north pole in the unit sphere $\bS^2\subset \bR^3$. This is
  \begin{itemize}
  \item analytic because its two components $\varphi$ and $(x,y)\mapsto x^2+y^2$ are;
  \item one-to-one because $\varphi$ already is;
  \item proper because $(x,y)\mapsto x^2+y^2$ is.
  \end{itemize}
  We thus have a closed analytic embedding of the form \Cref{eq:analyticmn}, with $m=2>1=n$.
\end{example}

In the proof of \Cref{th:cihg} we will, at one point, have to work with connected Lie subgroups $M\le G$ of a semisimple Lie group that are {\it maximal} among proper connected Lie subgroups. These have been studied extensively by Mostow in \cite{mst-max}, which deals mostly the case of linear $G$.

Although this is not stated explicitly in \cite{mst-max} (as far as I can tell), such maximal subgroups are always automatically closed, regardless of whether or not $G$ is linear. This follows by assembling together a number of remarks.

\begin{itemize}
\item Let us focus for the moment on the linear-$G$ case. As noted in loc.cit. (in the course of the proof of \cite[Theorem 3.1]{mst-max}), the Lie algebra
  \begin{equation}\label{eq:fmfg}
    \fm:=Lie(M)\subset \fg:=Lie(G)
  \end{equation}
  is maximal and hence {\it algebraic} in the sense of \cite[Definition 1]{chv-alg}. This follows, for instance, from the fact that Lie algebras have the same derived subalgebras as their algebraic hulls \cite[Proposition 1 3)]{chv-alg} and since $\fg$ is semisimple and hence coincides with its derived subalgebra \cite[\S 5.2, Corollary]{hum}, we cannot have
  \begin{equation*}
    \fm\subsetneq\text{algebraic hull of }\fm = \fg. 
  \end{equation*}
\item But then the Lie group corresponding to $\fm$ is expressible by polynomial equations (e.g. as explained on \cite[p.195]{chv-tuan}), so it will be closed.
\item All of that is still in the context of a {\it linear} semisimple $G$. Generally, given a maximal connected Lie subgroup $M\subset G$, the Lie-algebra inclusion \Cref{eq:fmfg} will stay as-is upon quotienting by a discrete central subgroup $D\subset G$ with $G/D$ linear (see the proof of \Cref{pr:max} for more on $D$). But we have just argued that
  \begin{equation*}
    MD/D\cong M/M\cap D\subset G/D
  \end{equation*}
  is closed, and hence so is the connected component $M=(MD)_0$ of its preimage through $G\to G/D$.
\end{itemize}

Henceforth, whenever handling maximal Lie subgroups (always of {\it semisimple} Lie groups), they will be assumed connected and proper. Their automatic closure will also be taken for granted, per the above remarks.

\pf{th:cihg}
\begin{th:cihg}
  There is no cost to assuming that $G$ and $H$ are Lie: $G$ has a normal compact subgroup $K\trianglelefteq G$ with $G/K$ Lie \cite[\S 4.6, Theorem]{mz}, and the passage from $H\le G$ to
  \begin{equation*}
    HK/K\cong H/H\cap K\le G/K
  \end{equation*}
  changes nothing ($HK\le G$ is still closed because $K$ is compact, and the characteristic indices do not change because again, we are modding out compact normal subgroups).

  The Lie-group version of the result, in turn, is now amenable to induction by $\dim(G)-\dim(H)$ (which quantity we refer to as the {\it dimension difference} of the inclusion). There is, of course, nothing to prove in the base case of dimension-difference $0$.

  So long as we can find an abelian, connected, proper and non-trivial normal subgroup $N\trianglelefteq G$ we can decompose the original inclusion as
  \begin{equation*}
    H\le \overline{HN}\le G.
  \end{equation*}
  Each of thee two successive inclusions has strictly smaller dimension difference, so we can appeal to the induction hypothesis assuming those inclusions have been taken care of. Note furthermore that the right-hand embedding
  \begin{equation*}
    \overline{HN}\le G
  \end{equation*}
  further reduces to
  \begin{equation*}
    \overline{HN}/N\le G/N
  \end{equation*}
  by \Cref{le:diff}. In this fashion, we can boil down the problem to two cases:
  \begin{enumerate}[(a)]
  \item\label{item:1} $G$ is of the form $\overline{HN}$ for a connected, normal, abelian group $N\trianglelefteq G$;
  \item\label{item:2} $G$ is semisimple.
  \end{enumerate}
  Case \Cref{item:1} we defer until later (\Cref{pr:timesrad}), noting here only that since $N$ is abelian and connected it must be of the form
  \begin{equation*}
    \mathrm{torus}\times \bR^n
  \end{equation*}
  (see for instance \cite[Chapter II, Exercise C.2]{helg}), and the torus component can always be annihilated with no change to characteristic indices. For that reason, when we return to \Cref{item:1} in \Cref{pr:timesrad}, we will be assuming that $N$ is a Euclidean group (i.e. one of the form $(\bR^n,+)$).

  The rest of the present proof, then, focuses on the semisimple-$G$ case (\Cref{item:2} above). The strategy will be to again shrink the dimension difference for as long as it is possible. Specifically, assuming $H\le G$ is not maximal among connected (proper) Lie subgroups, it can be embedded into such a maximal subgroup $M\le G$. We would then have to handle the two inclusions
  \begin{equation*}
    H\le M\quad\text{and}\quad M\le G
  \end{equation*}
  separately, given our induction hypothesis. For the former, we can simply proceed as before: if $M$ is not semisimple break up $H\le M$ into successive embeddings again resorting to induction, etc. As to the latter, we once more handle it separately as a special case in \Cref{pr:max}.
\end{th:cihg}

\subsection{Dense embeddings}\label{subse:dense}

Altering the notational lettering momentarily in order to avoid confusion later, consider a dense embedding $S\le G$ of connected Lie groups (not closed, in general: $\overline{S}=G$). According to \cite[Theorem 2.3]{2107.11796v1} we have $\mathrm{ci}(S)\ge \mathrm{ci}(G)$. It will be handy below to have a more careful estimate of the difference between the two characteristic indices.

\begin{theorem}\label{th:diffex}
  Let
  \begin{itemize}
  \item $S\le \overline{S}=G$ be a dense embedding of connected Lie groups;
  \item $K\le G$ a maximal compact subgroup with radical $A\le K$;
  \item $A_1$ the connected component $(A\cap S)_0$.
  \end{itemize}
  Given a decomposition
  \begin{equation}\label{eq:torrn}
    A_1\cong (\text{torus }\bT^m)\times \bR^n
  \end{equation}
  in the intrinsic topology on the Lie group $S$, we have
  \begin{equation*}
    \mathrm{ci}(S) - \mathrm{ci}(G) = n.
  \end{equation*}
\end{theorem}
\begin{proof}
  Dense embeddings of Lie groups are analyzed in enough detail in the proof of \cite[Theorem 1]{goto1} for us to be able to repurpose that argument.

  Consider, as in that proof, the universal cover
  \begin{equation*}
    \widetilde{G}\to G = \widetilde{G}/D
  \end{equation*}
  (where $D<\widetilde{G}$ is discrete and central, isomorphic to the fundamental group of $G$). Generally, tildes will adorn the connected components of preimages through this cover: we have $\widetilde{S}$, $\widetilde{A}$, $\widetilde{A_1}$, etc. $\widetilde{S}$. Note that
  \begin{itemize}
  \item $\widetilde{S}\le \widetilde{G}$ is closed because it is normal in a simply-connected Lie group \cite[p.127]{chv}.
  \item $\widetilde{A_1}\le \widetilde{G}$ is a closed Euclidean group, and having chosen a (closed, Euclidean) supplement for it in $\widetilde{A}$ in the sense that
    \begin{equation*}
      \widetilde{A}\cong \widetilde{A_1} \times \widetilde{A_2}
    \end{equation*}
    the product
    \begin{equation*}
      \widetilde{S}\times \widetilde{A_2}\ni (s,a)\mapsto sa\in \widetilde{S}\widetilde{A_2} = \widetilde{G}
    \end{equation*}
    is an isomorphism of analytic manifolds (not of groups, necessarily, because $\widetilde{S}$ and $\widetilde{A_2}$ need not commute). This, again, emerges as part of the proof of \cite[Theorem 1]{goto1}.
  \item per the discussion immediately preceding \cite[Theorem 1]{goto1}, we may as well assume that $D\le \widetilde{K}$ whence the torsion-free component $D_{free}$ in a decomposition
    \begin{equation}\label{eq:ddec}
      D \cong D_{tors}\times D_{free} := (\text{torsion})\times \bZ^{\mathrm{rank}(D)}
    \end{equation}
    embeds in $\widetilde{A}$.
  \end{itemize}
  Through a double application of \Cref{le:diff} we have
  \begin{align*}
    \mathrm{ci}(G) &= \mathrm{ci}(\widetilde{G})-\mathrm{rank}(D)\\
                   &=\mathrm{ci}(\widetilde{S}) + \dim \widetilde{A_2}-\mathrm{rank}(D)\\
                   &=\mathrm{ci}(S)+\mathrm{rank}(D\cap\widetilde{S})+ \dim\widetilde{A_2}-\mathrm{rank}(D),\\
  \end{align*}
  so the goal is to argue that
  \begin{equation}\label{eq:ddsa}
    \mathrm{rank}(D) = \mathrm{rank}(D\cap\widetilde{S}) + \dim\widetilde{A_2} + n
  \end{equation}
  for $n$ as in \Cref{eq:torrn}. A first observation is that since
  \begin{itemize}
  \item we are assuming that the free abelian summand in \Cref{eq:ddec} is a subgroup of $\widetilde{A}$;
  \item and quotienting by $D$ turns the Euclidean group $\widetilde{A}$ into a torus of the same dimension,
  \end{itemize}
  we have
  \begin{equation*}
    \mathrm{rank}(D) = \dim\widetilde{A} = \dim\widetilde{A_1}+\dim\widetilde{A_2},
  \end{equation*}
  and hence the target equation \Cref{eq:ddsa} becomes
  \begin{equation*}
    \dim\widetilde{A_1} = \mathrm{rank}(D\cap \widetilde{S}) + n. 
  \end{equation*}
  To prove this last equality, notice that on the one hand the left-hand side $\dim\widetilde{A_1} = \dim A_1$ is exactly the $m+n$ of \Cref{eq:torrn}, while on the other, given the notation \Cref{eq:ddec}, we have
  \begin{equation*}
    \mathrm{rank}(D\cap \widetilde{S}) = \mathrm{rank}(D_{free}\cap \widetilde{A}) = \mathrm{rank}(D_{free}\cap \widetilde{A_1}).
  \end{equation*}
  This is nothing but $\dim(A_1)=m$, and we are done.
\end{proof}

\begin{proposition}\label{pr:timesrad}
  Let $H\le G$ be a closed embedding of connected Lie groups with $G=\overline{HN}$ for a normal subgroup
  \begin{equation*}
    \bR^d\cong N\trianglelefteq G.
  \end{equation*}
  We then have the inequality \Cref{eq:cihg}: $\mathrm{ci}(H)\le \mathrm{ci}(G)$. 
\end{proposition}
\begin{proof}
  In passing from $H\le \overline{HN}$ to
  \begin{equation*}
    H/H\cap N\le \overline{HN}/H\cap N
  \end{equation*}
  an application of \Cref{le:diff} (or a double application, rather) shows that the two characteristic indices decrease by the same amount, so we may as well assume that $H\cap N$ is trivial.

  We apply \Cref{th:diffex} to the dense inclusion
  \begin{equation*}
    S:=HN\le G,
  \end{equation*}
  retaining the notation therein (for the groups $K$, $A$, $A_1$, etc.). Consider a decomposition \Cref{eq:torrn} for $A_1\le S=HN$. Since $A$ is a torus and $H\le G$ is closed, the torus component $\bT^m$ of that decomposition must contain the connected component $(H\cap A)_0$. Passing to Lie algebras, it follows in particular that the Euclidean component
  \begin{equation*}
    \bR^n\le \fa_1:= Lie(A_1)
  \end{equation*}
  of (the Lie-algebra version of) \Cref{eq:torrn} intersects
  \begin{equation*}
    \fh:=Lie(H)\subset Lie(HN)\cong \bR^d\rtimes \fh
  \end{equation*}
  trivially. It is a simple matter to prove, then, that the projection
  \begin{equation*}
    \bR^n\subset \fa_1\le \fh\subset \bR^d\rtimes \fh\to \bR^d
  \end{equation*}
  (linear but not, in general, a Lie-algebra morphism) is on-to-one, whence $n\le d$. We now have
  \begin{align*}
    \mathrm{ci}(G) &= \mathrm{ci}(HN) - n\\
                   &=\mathrm{ci}(\bR^d\rtimes H) - n\\
                   &=d+\mathrm{ci}(H) - n\\
                   &\ge\mathrm{ci}(H),
  \end{align*}
  where the first equality uses \Cref{th:diffex} (as indicated, with $S=HN\cong \bR^d\rtimes H$), the third is an application of \Cref{le:diff}, and the last inequality is the above remark that $n\le d$.
\end{proof}

To elucidate the phenomenon that underpins \Cref{pr:timesrad}, some comments and examples are perhaps in order. Assume, as done at the start of the proof of \Cref{pr:timesrad}, that $H$ intersects $N$ trivially. Abstractly, with its intrinsic topology (rather than the subspace topology inherited from $G$), $HN$ is then isomorphic to the semidirect product $\bR^d\rtimes H$ with respect to the adjoint action of
\begin{equation*}
  H\le G\text{ on }\bR^d\cong N\le G. 
\end{equation*}
On the one hand, according to \Cref{le:diff} we have
\begin{equation}\label{eq:hnd}
  \mathrm{ci}(HN) = \mathrm{ci}(\bR^d\rtimes H) = d+\mathrm{ci}(H).
\end{equation}
On the other hand though, by \cite[Theorem 2.3]{2107.11796v1}, in passing to $G=\overline{HN}$ we have to then adjust the characteristic index {\it down} from that value because the embedding $HN\le G$ is dense:
\begin{equation*}
  \mathrm{ci}(HN)\ge \mathrm{ci}(G). 
\end{equation*}
The point of \Cref{pr:timesrad}, though, is that because $H\le G$ was closed, this latter discrepancy resulting from the dense embedding cannot be larger than the $d$ we originally supplemented $H$ with in \Cref{eq:hnd}. That dense embeddings can (in this regard) be pathological enough to achieve this upper bound can be illustrated with an example adapted from \cite[Appendix]{goto1} (used there for different but related purposes).

\begin{example}\label{ex:hgh}
  We want a dense (connected-)Lie-group embedding $\bR^d\rtimes H\le G$ with $H$ and $\bR^d$ both closed in $G$ and such that
  \begin{equation}\label{eq:gish}
    \mathrm{ci}(G) = \mathrm{ci}(H).
  \end{equation}
  It will be enough to do this for $d=1$, as the $d^{th}$ Cartesian power of {\it that} example will then handle the general case.

  Consider, as on \cite[p.118]{goto1}, the universal cover
  \begin{equation*}
    1\to \bZ=\langle \sigma\rangle\longrightarrow \widetilde{SL(2,\bR)}\longrightarrow SL(2,\bR)\to 1.
  \end{equation*}
  The relevant objects are
  \begin{itemize}
  \item $H:=\widetilde{SL(2,\bR)}$;
  \item acting trivially on $N:=\bR$;
  \item and $G:=H\times N\times \bR/D$ with
    \begin{equation*}
      D:=\{(\sigma^{m+n},\ m+n,\ m+n\gamma)\in G\ |\ m,n\in \bZ\}
    \end{equation*}
    for an irrational $\gamma\in \bR$. 
  \end{itemize}
  As noted in loc.cit., $HD/H$ and $ND/D$ (easily seen to be isomorphic to $H$ and $N$ respectively) are both closed in $G$, while their product is not. It follows, for dimension reasons, that we must have $\overline{HN}=G$.

  As to characteristic indices, note first that $\mathrm{ci}(H)=3$, since in fact $H=\widetilde{SL(2,\bR)}$ is homeomorphic to $\bR^3$. On the other hand, $G$ is obtained from $H\times \bR^2$ by quotienting out a discrete (closed, central) subgroup $D\cong \bZ^2$, meaning that by \Cref{le:diff} its characteristic index is
  \begin{equation*}
    \mathrm{ci}(G) = \mathrm{ci}(H) + \mathrm{ci}(\bR^2) - \mathrm{rank}(\bZ^2) = \mathrm{ci}(H);
  \end{equation*}
  \Cref{eq:gish}, in other words.
\end{example}

\subsection{Characteristic indices of maximal subgroups}\label{subse:max}

In the discussion below diagrams of the form
\begin{equation}\label{eq:efb}
  \begin{tikzpicture}[auto,baseline=(current  bounding  box.center)]
    \path[anchor=base] 
    (0,0) node (l) {$F$}
    +(1,0) node (u) {$E$}
    +(1.5,-1) node (r) {$B$}
    ;
    \draw[->] (l) to[bend left=6] node[pos=.5,auto] {$\scriptstyle $} (u);
    \draw[->] (u) to[bend left=6] node[pos=.5,auto] {$\scriptstyle $} (r);
  \end{tikzpicture}
\end{equation}
indicate (locally trivial) fibrations \cite[\S 2]{stn-fib} with total space $E$, fiber $F$ and base $B$. We chain several of these together (as will become apparent) to indicate that the fibers themselves are total spaces of further fibrations. 

The fibrations we are concerned with here will be at least locally trivial with everything in sight a Hausdorff, metrizable topological manifold, such as, say, 
\begin{equation}\label{eq:mggm}
  \begin{tikzpicture}[auto,baseline=(current  bounding  box.center)]
    \path[anchor=base] 
    (0,0) node (l) {$M$}
    +(1,0) node (u) {$G$}
    +(1.5,-1) node (r) {$M/G$}
    ;
    \draw[->] (l) to[bend left=6] node[pos=.5,auto] {$\scriptstyle $} (u);
    \draw[->] (u) to[bend left=6] node[pos=.5,auto] {$\scriptstyle $} (r);
  \end{tikzpicture}
\end{equation}
for any closed subgroup $M\le G$ of a Lie group (that this is indeed a fibration follows, for instance, from \cite[\S 7.5]{stn-fib} or \cite[Corollary 2]{mst}). They are in particular {\it Serre fibrations} in the sense of \cite[Chapter VII, Definition 6.2]{brd-tg}, i.e. the `espaces fibr\'es' of \cite[Chapitre II, \S 2, D\'efinition]{ser-fib} (as mentioned in \cite[Chapitre II, \S 2, Exemples]{ser-fib}), so the results of this latter source apply.

\begin{proposition}\label{pr:diminv}
  Consider a topological manifold $M$ fitting into a chain
  \begin{equation}\label{eq:chain}
    \begin{tikzpicture}[auto,baseline=(current  bounding  box.center)]
      \path[anchor=base] 
      (0,0) node (m1) {$M_1$}
      +(1,0) node (m) {$M$}
      +(1.5,-1) node (b0) {$B_0$}
      +(0.5,-1) node (b1) {$B_1$}
      +(-1,0) node (d) {$\cdots$}
      +(-2,0) node (mn1) {$M_{n-1}$}
      +(-1.5,-1) node (bn1) {$B_{n-1}$}
      +(-3.5,0) node (r) {$\bR^{d}$}
      ;
      \draw[->] (m1) to[bend left=6] node[pos=.5,auto] {$\scriptstyle $} (m);
      \draw[->] (m) to[bend left=6] node[pos=.5,auto] {$\scriptstyle $} (b0);
      \draw[->] (m1) to[bend left=6] node[pos=.5,auto] {$\scriptstyle $} (b1);
      \draw[->] (mn1) to[bend left=6] node[pos=.5,auto] {$\scriptstyle $} (bn1);
      \draw[->] (r) to[bend left=6] node[pos=.5,auto] {$\scriptstyle $} (mn1);
      \draw[->] (d) to[bend left=6] node[pos=.5,auto] {$\scriptstyle $} (m1);
    \end{tikzpicture}
  \end{equation}
  of manifold fibrations with all manifolds connected and all bases $B_i$ compact. Then, $d$ can be recovered as
  \begin{equation}\label{eq:dmz2}
    d = \dim M - (\text{largest $m$ with }H_m(M,\bZ/2)\ne 0).
  \end{equation}  
\end{proposition}
\begin{proof}
  In general, for a manifold $X$ and coefficients $R$ (an abelian group, or ring, or system of local coefficients \cite[Appendic 3.H]{htch}, etc.; whatever is appropriate), we will write
  \begin{equation*}
    H_{max}(X,R)\text{ or }H_{max=k}(X,R)
  \end{equation*}
  for the largest non-zero homology group of $X$ valued in $R$, with the more elaborate notation also indicating inline what that maximal index is (namely $k$).
  
  Consider a fibration \Cref{eq:efb} of connected topological manifolds, with $B$ compact and such that
  \begin{equation*}
    H_{max=k}(F,\bZ/2)\cong \bZ/2. 
  \end{equation*}
  It then follows from
  \begin{itemize}
  \item the fact that
    \begin{equation*}
      H_{max=\dim B}(B,\bZ/2)\cong \bZ/2
    \end{equation*}
    (\cite[Chapter VI, Corollary 7.12]{brd-tg});
  \item together with the Serre spectral sequence
    \begin{equation*}
      E^2_{p,q}:=H_p(B,H_q(F,\bZ/2))\Rightarrow H_{p+q}(E,\bZ/2)
    \end{equation*}
    attached to the fibration (\cite[Chapitre II, \S 2, Th\'eor\`eme 2]{ser-fib})
  \end{itemize}
  that
  \begin{equation*}
    H_{max=k+\dim B}(E,\bZ/2) \cong \bZ/2. 
  \end{equation*}
  Applying this remark recursively, starting with the leftmost fibration 
  \begin{equation*}
    \begin{tikzpicture}[auto,baseline=(current  bounding  box.center)]
      \path[anchor=base] 
      (0,0) node (l) {$F$}
      +(1,0) node (u) {$E$}
      +(1.5,-1) node (r) {$B$}
      ;
      \draw[->] (l) to[bend left=6] node[pos=.5,auto] {$\scriptstyle $} (u);
      \draw[->] (u) to[bend left=6] node[pos=.5,auto] {$\scriptstyle $} (r);
    \end{tikzpicture}
  \end{equation*}
  in \Cref{eq:chain} and proceeding rightward, we obtain
  \begin{equation*}
    H_{max=\dim M-d}(M,\bZ/2)\cong \bZ/2.
  \end{equation*}
  As this is in fact an enhancement of the sought-after conclusion \Cref{eq:dmz2}, we are done.
\end{proof}

\begin{proposition}\label{pr:max}
  For a maximal proper Lie subgroup $M\le G$ of a connected semisimple Lie group $G$ we have
  \begin{equation*}
    \mathrm{ci}(M)\le \mathrm{ci}(G). 
  \end{equation*}
\end{proposition}
\begin{proof}
  First, we reduce to the case of linear $G$ by modding out a discrete central subgroup $D\trianglelefteq G$:
  \begin{itemize}
  \item the universal cover $\widetilde G$ has a linear quotient by a discrete central subgroup $D_1\trianglelefteq \widetilde G$ \cite[\S 1.4]{ragh};
  \item whereas $G$ is a quotient of $\widetilde G$ by some other discrete central subgroup $D_2\trianglelefteq \widetilde G$;
  \item whereupon we can take
    \begin{equation*}
      D:=D_1D_2/D_2\triangleleft G=\widetilde{G}/D_2:
    \end{equation*}
    the quotient $G/D\cong \widetilde{G}/D_1D_2$ of the semisimple linear Lie group $\widetilde{G}/D_1$ will automatically be linear \cite[Lemma 9]{goto2}.
  \end{itemize}
  By \Cref{le:diff} we have
  \begin{equation*}
    \mathrm{ci}(G)-\mathrm{ci}(G/D)=\mathrm{rank}(D)
  \end{equation*}
  and
  \begin{equation*}
    \mathrm{ci}(M)-\mathrm{ci}(MD/D)=\mathrm{rank}(M\cap D),
  \end{equation*}
  so subtracting the two and using the obvious inequality $\mathrm{rank}(D)\ge \mathrm{rank}(M\cap D)$ we obtain
  \begin{equation*}
    \mathrm{ci}(G)-\mathrm{ci}(M)\ge \mathrm{ci}(G/D)-\mathrm{ci}(MD/D). 
  \end{equation*}
  In other words, if the conclusion holds for the embedding $MD/D\le G/D$ of linear groups then it holds in its original form. For that reason, we will henceforth assume that everything in sight is linear. 
  
  \cite[Theorem 3.1]{mst-max} then applies, ensuring that either
  \begin{itemize}
  \item the radical $R$ of $M$ is compact;
  \item or the homogeneous space $G/M$ is compact.
  \end{itemize}
  In the former case ($M$ has compact radical $R$) we have
  \begin{equation*}
    \mathrm{ci}(M)=\mathrm{ci}(M/R) = \mathrm{ci}(L)
  \end{equation*}
  for a Levi factor $L\le M$ \cite[\S 1.3]{ragh} and the problem reduces to the inclusion $L\le G$ where the smaller group $L$ is also semisimple. The desired conclusion is now precisely \cite[Proposition 2.4]{2107.11796v1}.

  This leaves the case when $G/M$ is compact. We then have, on the one hand, the fibration \Cref{eq:mggm} with compact base and fiber $M\cong K_M\times \bR^{\mathrm{ci}(M)}$ for a maximal compact subgroup $K_M\le M$, and on the other the analogous decomposition $G\cong K_G\times \bR^{\mathrm{ci}(G)}$. The fact that the ``non-compact piece'' must have the same dimension
  \begin{equation*}
    \mathrm{ci}(G) = \mathrm{ci}(M)
  \end{equation*}
  now follows from \Cref{pr:diminv}.
\end{proof}


\def\polhk#1{\setbox0=\hbox{#1}{\ooalign{\hidewidth
  \lower1.5ex\hbox{`}\hidewidth\crcr\unhbox0}}}

\addcontentsline{toc}{section}{References}

\Addresses

\end{document}